\documentclass[11pt]{amsart}
\usepackage{pifont}
\usepackage{tikz}
\usepackage[utf8]{inputenc}
\usepackage[T1]{fontenc}
\usepackage{amsmath,amscd,amssymb}
\newtheorem{theorem}{Theorem}[section]
\newtheorem{conjecture}[subsection]{Conjecture}
\newtheorem{lemma}[theorem]{Lemma}
\newtheorem{proposition}[theorem]{Proposition}
\newtheorem{remark}[theorem]{Remark}
\numberwithin{equation}{section}
\newtheorem{definition}[theorem]{Definition}

\input xy
\xyoption{all}

\begin{document}
\title[Configuration of zero-dimensional subschemes and Mestrano-Simpson conjecture] {Geometry of  some moduli  of  bundles over a very general sextic surface for small second Chern classes and Mestrano-Simpson Conjecture}

\author[D. Bhattacharya]{Debojyoti Bhattacharya }
\address{Indian Institute of Science Education and Research,Thiruvananthapuram,
Maruthamala PO, Vithura,
Thiruvananthapuram - 695551, Kerala, India}
\email{debojyotibhattacharya15@iisertvm.ac.in}

\author[S.Pal]{Sarbeswar Pal}

\address{Indian Institute of Science Education and Research,Thiruvananthapuram,
Maruthamala PO, Vithura,
Thiruvananthapuram - 695551, Kerala, India.}

\email{spal@iisertvm.ac.in, sarbeswar11@gmail.com}

\keywords{vector bundles,   Cayley-Bacharach property, Mestrano-Simpson conjecture}

\date{}

\begin{abstract}
Let $S \subset \mathbb P^3$ be a very general sextic surface over complex numbers.  Let  $\mathcal{M}(H, c_2)$  be the moduli space of  rank $2$ stable bundles  on $S$ with fixed first  Chern class $H$ and second Chern class $c_2$. In this article  we study the configuration of certain locally complete intersection zero-dimensional subschemes on $S$ satisfying Cayley-Bacharach property, which leads to the existence of non-trivial sections of a general memeber of each component of the moduli space for small $c_2$. Using this study  we will make an attempt to prove Mestrano-Simpson conjecture on the number of irreducible components of $\mathcal{M}(H, 11)$ and prove the conjecture partially. We will also   show that $\mathcal{M}(H, c_2)$ is irreducible for $c_2 \le 10$ .
\end{abstract}
\maketitle\section{Introduction}

Let $S$ be a smooth projective irreducible  surface over $\mathbb C$  and $H$ be an ample divisor on S.
Let $r \ge 1$ be an integer, $L$ be a line bundle on $S$, and $c_2 \in H^4(S,\mathbb Z)\,\simeq \,\mathbb Z$.  The moduli space of semistable torsion free sheaves on $S$(w.r.t $H$) with fixed determinant $L$
and second Chern class $c_2$ was first constructed by Gieseker and Maruyama (see \cite{GM}, \cite{MAR}) using Mumford's geometric invariant theory 
and it was proved that  it's a projective scheme 
(need not be reduced). After their construction many people have studied 
the geometry of this moduli space. The study has been done by fixing the underlying surface. For example, when the surface is rational Barth \cite{WB}, Costa and Rosa Maria \cite{RM},  Le Potier \cite{LP} 
proved that the moduli space is reduced, irreducible and rational under certain conditions on rank and Chern classes. When the surface is K3 it has been studied by Mukai \cite{SM} and many others.  When the surface is general, Jun Li \cite{JL} showed that for $c_2$ big enough, the moduli space is also of general type. The guiding general philosophy  
is that the geometry of the moduli space is reflected by the underlying geometry of the surface. The first result without fixing the underlying surface was given by O'Grady. In \cite{OG} O'Grady proved that for sufficiently large second Chern class $c_2$, the moduli space is reduced, generically smooth and irreducible. In fact O'Grady's first step to prove irreducibility was to show each component is generically smooth of expected dimension. The generic smoothness result was also proved by Donaldson \cite{D} in the rank $2$  and trivial determinant case and Zuo \cite{ZUO} for arbitrary determinant.

After O'Grady's result it was important to give an effective bound on $c_2$ for the irreducibility and generic smoothness of the moduli space.  The moduli space of vector bundles over hypersurfaces is one of the important objects to study.
When the underlying surface $S$ is a very general quintic hypersurface in $\mathbb P^3$  Mestrano and Simpson  studied this question systematically and in \cite{KS}, the second  author of this article with K. Dan, studied the question related to Brill-Noether loci.   In a series of papers \cite{SIM1}, \cite{SIM1.5}, \cite{SIM2}, \cite{SIM3},  Mestrano and Simpson proved that the moduli space of rank $2$, $H-$stable torsion free sheaves with fixed determinant $H$ is generically smooth and  irreducible, where $H:= \mathcal{O}_S(1)$. This result was known before by an unpublished work by Nijsse for $c_2 \ge 16$ \cite{N}.

Motivated by the results of Mestrano and Simpson  we look at the next case i.e. the moduli space of rank $2$ torsion free sheaves on a very general sextic surface $S$, that is,  a very general hypersurface  of degree $6$  in 
$\mathbb{P}^3$. 

In \cite{SIM2},  Mestrano and Simpson  showed that the moduli space of stable rank $2$ bundles over a very general sextic surface 
is not irreducible for $c_2 =11$. In fact, they have shown that the moduli space in this case has at least  two different irreducible components.

In fact, they constructed a $12$-dimensional irreducible component $\mathcal{M}_1$ consisting of  vector bundles fitting in an exact sequence of the form
\[
0 \to \mathcal{O}_S \to E \to \mathcal{J}_P(1) \to 0,
\]
where $P$ is a  zero-dimensional locally complete intersection subscheme contained in a rational cubic curve and  a family of vector bundles $\mathcal{M}_2$ of dimension at least $13$ consisting of vector bundles  fitting in an exact sequence of the form
\[
0 \to \mathcal{O}_S \to E \to \mathcal{J}_P(1) \to 0,
\]
where $P$ is a  zero-dimensional locally complete intersection subscheme contained in a hyperplane section. 
  Then they conjectured that $\mathcal{M}_1$ and $\mathcal{M}_2$   are the only two components. In other words, $\mathcal{M}_2$ is an irreducible component of dimension at least $13$ and   $\mathcal{M}_1$ and $\mathcal{M}_2$ cover $\mathcal{M}(H, 11)$, where $\mathcal{M}(H,c_2)$ denotes the moduli space of
$H$-stable, rank $2$ locally free sheaves on $S$ with fixed determinant isomorphic to $H$ and  second Chern class  $c_2$. \\ 

Our main goal in this article is to prove that a general point of each of irreducible components of $\mathcal M(H,c_2)$ over a very general sextic surface with second Chern class $\le 18$ admits a section. We
 apply it  to give a partial proof of  the above conjecture made by  Mestrano and Simpson. 
 Further more the natural questions one can ask are the following:\\
$(1)$ Is  the moduli space  irreducible for $c_2 \le 10$ ?\\
$(2)$  Can one give an effective bound for $c_2$ such that the moduli space
becomes irreducible ?

Let $S \subset \mathbb{P}^3$ be a very general sextic surface over $\mathbb C$ and $H$ denote the very ample line bundle $\mathcal O_S(1)$. The Picard group of $S$ is generated by $H$. Let  $\mathcal{M}(H,c_2)$ denote the moduli space of
$H$-stable, rank $2$ locally free sheaves on $S$ with fixed determinant isomorphic to $H$ and  second Chern class  $c_2$ and $\overline{\mathcal{M}}(H, c_2)$ be  the Gieseker-Maruyama moduli space of semistable torsion free sheaves. 

It is known that $\overline{\mathcal{M}}(H, c_2)$ is projective and $\mathcal{M}(H, c_2)$ sits inside
$\overline{\mathcal{M}}(H, c_2)$
as an open subset, whose complement is called the \textit{boundary}.

The very first and main step towards a proof of the conjecture is  to show that a general element in  each of irreducible components $\mathcal{M}(H, c_2)$ admits a non-zero section.   Major part of this article is devoted to prove the above fact.  The main idea to prove that fact  is to  investigate the possible configurations of certain locally complete intersection zero-dimensional subschemes on $S$ satisfying Cayley-Bacharach property. 
 Using it we also give a bound for the boundary strata of the moduli space. More precisely, we will prove the following Theorems.
 
 \begin{theorem}\label{PP1}
  Let $X$ be an irreducible component of $\mathcal{M}(H, c_2)$. If $c_2 \le 18$, then  a general element $E \in X $ admits a non-zero section. 
   \end{theorem}

\begin{theorem}
The moduli space $\mathcal{M}(H, c_2)$ is non-empty for $c_2 \ge 5, c_2 \ne 7$ and it is irreducible for $c_2 \le 10$.
\end{theorem}

\begin{theorem}
Suppose $\mathcal{M}(H, c_2)$ is good for $c_2 \ge 27$. Then $\overline{\mathcal{M}}(H, c_2)$ is also good for $c_2 \ge 27$. 
\end{theorem}
Finally  as an application we will give a partial proof of  the Mestrano-Simpson conjecture, more precisely we will prove that $\mathcal{M}_1, \mathcal{M}_2$ cover $\mathcal{M}(H, 11)$
and $\mathcal{M}_2$ is of dimension exactly $13$ containing an irreducible subset of dimension $13$ and possibly one more irreducible component of dimension $12$.

  In particular, $h^0(E)> 0$ for any stable bundle $E \in \mathcal{M}(H, 11)$. 

\subsection{ The organization of the paper:}

In section \ref{S1}, we will recall some basic results which we will use in the subsequent sections. \\
In section \ref{S2},  we investigate the possible configurations of a  zero dimensional subscheme  of length $l \ge 17$, satisfying Cayley-Bacharach property for $\mathcal{O}(5)$.  In fact we shall show that if $l \le 16$ then any such subscheme is contained in a quadric hypersurface.\\
In section \ref{S3}, we will show that if $c_2 \le 18$, then every general  point $E$ in each of the irreducible components of $\mathcal{M}(H, c_2)$  admits a  non-zero section. \\
In section \ref{S4}, we apply Theorem \ref{PP1} to give a partial proof of  Mestrano-Simpson conjecture. \\
In section \ref{S5}, we will show that the moduli space $\mathcal{M}(H, c_2)$ is non-empty and irreducible for $5 \le c_2 \le 10, c_2 \ne 7$ and for $c_2 \le 4$ and $c_2 = 7$, it is empty.\\
Finally in section \ref{S6}, we will give an upper bound for the boundary strata and using it we will show that if $\mathcal{M}(H, c_2)$ is good for $c_2 \ge 27$, then $\overline{\mathcal{M}}(H, c_2)$ is also good for $c_2 \ge 27$. In other words, every component of $\overline{\mathcal{M}}(H, c_2)$ has the dimension equals to the expected dimension $4c_2 -39$.

\section*{Notation and convention}
For   the line bundle $\mathcal{O}_{\mathbb{P}^3}(n)$ we write simply   $\mathcal{O}(n)$ and for a subscheme $Z \subset \mathbb{P}^3$, we denote the pull back of $\mathcal{O}(n)$ to $Z$  by $\mathcal{O}_Z(n)$. \\
If  $S$ is a very general hypersurface of degree $6$ then, $\text{Pic}(S) \simeq \mathbb{Z}$.   It is not difficult to see that $h^0(S,  \mathcal{O}_S(n)) = h^0(\mathbb{P}^3, \mathcal{O}(n)), n \le 5$. Thus a zero-dimensional subscheme  $P \subset S$ which satisfies Cayley-Bacharach property for $\mathcal{O}_S(m)$  also satisfies Cayley-Bacharach property for $\mathcal{O}(m)$   and vice-versa.\\
Thus if  a zero-dimensional subscheme  $P \subset S$ satisfies Cayley-Bacharach property for $\mathcal{O}_S(m)$, then with out loss of generality we can assume  that   $P $  satisfies Cayley-Bacharach property for $\mathcal{O}(m)$ in $\mathbb{P}^3$ and  we write $P$ satisfies $\text{CB}(m)$.\\ 
%We will always denote the length a zero dimensional scheme $P$ by $|P|$. 

\section{Preliminaries}\label{S1}
In this section we will recall few results which we need in next sections. 
\begin{theorem}\label{T1}
Let $P$ be a set of points in $\mathbb{P}^r$, and let $d \ge 2$ be an integer. If, for all $k \ge 1$, no $dk + 2$ points of $P$ lie in a projective $k-$plane,
then $P$ impose independent conditions on forms of degree $d$; in fact there is a multilinear form of degree $d$ containing any subset consisting of all but one of the 
points, but missing the last.
\end{theorem}
\begin{proof}
See \cite[Theorem 2]{DE}
\end{proof}
\begin{theorem}\label{T}(Chasles)
If a set $\Gamma_1$ of $8$ points in $\mathbb{P}^2$ lies in the complete intersection
$\Gamma$ of two cubics, then any cubic vanishing on $\Gamma_1$ vanishes on $\Gamma$.
\end{theorem}
\begin{proof}
See \cite[Corollary 2.8]{EP}.
\end{proof}
\begin{remark}\label{R}
Note that any $m$ points of a plane lying in an intersection of two  cubics in $\mathbb P^3$, where $m \ge 9$, satisfies the Cayley-Bacharach property for $\mathcal{O}(3)$ in $\mathbb{P}^3$. 
 \end{remark}
\begin{theorem}[Castelnuovo Theorem]\label{CN}
Let $\Gamma \subset \mathbb{P}^n$ be a zero dimensional subscheme of degree $d$ which are in linearly general position. Denote $h_{\Gamma}(m)$ by the number of independent conditions imposed by $\Gamma$ on degree $m$ hypersurfaces. If $d \ge 2n+3$ and $h_{\Gamma}(m) = \text{min}\{d, mn+1\}$, then $\Gamma$ lies in a raional normal curve.
\end{theorem}
\begin{proof}
See \cite{ACGH}
\end{proof}

   \section{Configurations of certain zero dimensional subscheme of $S$ satisfying $\text{CB}(5)$}\label{S2}

   Let $S$ be a very general,  irreducible hypersurface of degree $6$ in $\mathbb{P}^3$.   Let $P \subset S$ be a zero dimensional subscheme  satisfying  $\text{CB}(5)$.  In this section we shall investigate the possible configurations of $P$.  
   More precisely, we prove the following Theorem,
   \begin{theorem}\label{TY}
   Let $S$ be a very general   hypersurface of degree $6$ in $\mathbb{P}^3$. Let $P$ be a zero-dimensional, locally complete intersection subscheme of $S$ of length $c_2 +12$,  $5 \le c_2  \le 18$, satisfying $\text{CB}(5)$ and not contained in any quadratic hypersurface. Then we have the following possibilities:\\
$(I)$ $P$ is contained in a union of a rational cubic curve and a pair of skew lines.\\ 
$(II)$  $P$ is contained in a union of a rational cubic curve, a line and a degree $2$ plane curve. \\
$(III)$ $P$ is contained in a union of a rational cubic curve and a plane curve of degree $2$.\\
$(IV)$ $P$ is contained in a union of a rational cubic curve and plane curve of degree $3$.\
 \end{theorem}
 \begin{proof}
   If a subscheme $P^{\prime}$ of $P$ of length $17$ fails to impose independent conditions on sections of $\mathcal{O}(5)$, then by Theorem \ref{T1}, there is a plane containing a subscheme of length at least $12$. Thus if any subscheme of $P$ of length $17$ fails to impose independent conditions on sections of $\mathcal{O}(5)$, then repeating the argument on several subschemes of length $17$, one can show that $P$ lie on a plane.\\
   Thus there is a subscheme $P^{\prime}$ of $P$ of length at least $17$ which imposes independent conditions on sections of $\mathcal{O}(5)$.\\
   
   {\bf Step-$1$}: \\
   In this step we will show that $P^{\prime}$ is contained in a union of a rational cubic curve and two planes not necessarily distinct.\\
   If $P^{\prime}$ is in linearly general position, then by repeated application of Theorem \ref{CN}, we see that $P^{\prime}$ is contained in a rational cubic curve. \\
   If $P^{\prime}$ is not in linearly general position, then there is a plane $H_1$ containing a subscheme $P^{\prime \prime}$ of length at least 4. If the residual of $P^{\prime \prime}$ in $P^{\prime}$ is not in linearly general position, then there is another plane $H_2$, containing a subscheme of length at least $4$.\\
   If the residual of $P^{\prime }\cap (H_1 \cup H_2)$ has length at least $4$, then it has to be in linearly general position. Otherwise, there will be another plane $H_3$ containing a subscheme of length at least $4$. Thus there are $3$ planes each contains a subscheme of length at least $4$. Then one can find two subschemes $P_1, P_2$ of length $9$ contained in $P^{\prime} \cap (H_1 \cup H_2 \cup H_3)$ which are in linearly general position and they intersect in a subscheme of length at least $7$.  Thus by Theorem \ref{CN}, $P_1$ and $P_2$ contained in  rational cubic curves $C_1$ and $C_2$ respectively and $C_1 \cap C_2$ has length at least $7$. Since two distinct rational cubic curves can intersect in a subscheme of length at most $6$, $C_1=C_2$, and hence it intersects one of the plane in a subscheme of length at least $4$, a contradiction. \\
   Thus $P^{\prime}$ is contained in a union of a rational cubic curve $C_1$  and two planes $H_1, H_2$.\\
   
   {\bf Step-$2$}:\\
   In this step we will show the subscheme $P$ itself is contained in a union of a rational cubic curve and two planes not necessarily distinct.
   Assume $H_1 \ne H_2$.
   Let $\tilde{P}$ be the residual subscheme of $P \cap (C_1 \cup H_1 \cup H_2)$
    in $P$. Then $\tilde{P}$ has length at most $13$. If $\tilde{P}$ fails to impose independent conditions on sections of $\mathcal{O}(5)$, then by Theorem \ref{T1}, they lie in a plane $H_3$. Then any subscheme of length $12$ intersecting each of the hyperplanes $H_i, 1= 1, 2, 3$, in a subscheme of length $4$, fails to impose independent conditions and hence they all lie in a plane and intersects all the planes $H_i, i = 1, 2, 3$ in a subscheme of length $4$.  Otherwise we will get a contradiction as in Step- $1$.
    Therefore, in this situation, $H_1=H_2=H_3$.\\
    Let us assume that $\tilde{P}$ imposes independent conditions on sections of $\mathcal{O}(5)$. If the subschemes $\tilde{P}$ and $P \cap C_1$, together are not in linearly general position, then there is a plane containing a subscheme of length at least $4$ and we get a contradiction as in Step-$1$. Hence they are in linearly general position. Then one can show that they all lie in $C_1$.\\
    Therefore, $P$ is contained in a union of a rational cubic curve $C_1$ and two planes $H_1, H_2$.\\
    
   {\bf Step-$3$}:\\
   In this step we will show that if the two planes are distinct, then only the first two cases in the Theorem can  occur.
   Note that, the subschemes $H_i \cap P, i = 1. 2$ satisfy $\text{CB}(2)$, hence they have  length at least $4$.
   Since, each  of the subschemes in $C_1\cap P$ satisfies $\text{CB}(3)$, the length of the subscheme $C_1 \cap P$ is at least $11$.\\
   Case-$(1)$:   Length of $C_1 \cap P \le 16$.\\
   Let us assume that both the subschemes  $P \cap H_1$ and $ P \cap H_2$ lie in curves of degree $2$.  Note that, a general quintic hypersurface intersects a degree two curve in a subscheme of length $10$. Thus if both the
   subschemes  $P \cap H_1$ and $ P \cap H_2$ have length $\le 10$, then one can find a quintic hypersurface containing the subschemes $P \cap H_1$, $ P \cap H_2$ and a co-length $1$ subscheme of  $C_1 \cap P$ but not containing the degree two curves and  $C_1 \cap P$. \\
   Case-$(2)$: Length of $C_1 \cap P \ge 17$.\\
   In this case one of the subschemes, say $P \cap H_1$ has length $\le 6$. Then one can find a quintic hypersurface containing $C_1$, the subscheme $P \cap H_2$ and a co-length $1$ subscheme of $P \cap H_1$ but not containing $P \cap H_1$ .\\
   Thus $P$ fails to satisfy $\text{CB}(5)$. \\
   Let one of the subschemes say, $P \cap H_1$ has length at least $11$. Then the other subscheme $P \cap H_2$ has length at most $8$. Note that, the degree $2$ curve containing $P \cap H_1$ can intersect the other degree two curve containing $P \cap H_2$ in a subscheme of length at most two, as earlier one can find a quintic hypersurface, containing $C_1 \cap P$ but not containing $C_1$, the degree two curve containing $P \cap H_1$ and a co-length $1$ subscheme of $P \cap H_2$ but not containing $P \cap H_2$. Hence $P$ fails to satisfy $\text{CB}(5)$. 
   Thus one of the subschemes $H_1 \cap P$ and $H_2 \cap P$ lie in a line and other lie on a degree two plane curve.\\
   
   {\bf Step-$4$}:\\
   In this step we will show that if the two planes coincide, then only last cases in the Theorem can occur.
   Let $H_1= H_2= H$. Then since $P \cap C_1$ satisfies CB(4), its length is at least $14$. Thus the length of $P \cap H$ is at most $16$. If $P \cap H$ lie on a line, then it has length at most $6$. Then one can find a quadratic hypersurface containing $C_1$ and a subscheme of $P \cap H$ of length $2$ and the residual has length at most $4$ and  satisfy $\text{CB}(3)$, a contradiction.  Thus $P \cap H$ is contained in curve of degree $\ge 2$. If it is contained in a curve of degree $\ge 4$, then as in Step-$4$, one can show that, $P$ fails to satisfy $\text{CB}(5)$.

\end{proof}
    
 \begin{remark}\label{R11}
  Assume that the case $(II)$ occurs, that is $P$ is contained in a union of rational cubic curve $C_1$, a line $l$ and a degree $2$ plane curve
 $C_2$. Then one can get a quadratic hypersurface $Q$ containing $C_1$ and a subscheme of length $2$ of $P \cap l$ and residual of $Q$ and the plane containing $C_2$ satisfies $\text{CB}(2)$, hence has length at least $4$. Thus $P \cap l$ has length   $6$. Similarly, the length of $P \cap C_2$ has length at least $7$. If $C_2 \cap P$
 has length $=d+7$ and $C_1 \cap P$ has length at most $16-d$ for some non-negative integer $d$, then one can see that $P$ fails to satisfy $\text{CB}(5)$. Thus in that case the length of $P$ has to be at least $30$. Similarly if the case $(IV)$ occurs, then one can show that the length of $P$ has to be at least $30$.\\ 
 
 Let us assume that case $(I)$ occurs, i.e., $P$ is contained in a union of a rational cubic $C_1$ and a pair of skew lines $l_1$ and $l_2$. As in the previous case we have, $l_i \cap P$ has length $6$ for $i= 1, 2$. Since a line intersects a rational cubic in a subscheme of length at most $2$, one can show that if $P \cap C_1$ has length at most $13$, then $P$ fails to satisfy $\text{CB}(5)$. On the other hand, if $P \cap C_1$ has length $=14$, then $P$ satisfy $\text{CB}(5)$ only if $C_1$ intersects both the lines $l_1$ and $l_2$ in a subscheme of length $2$.\\
 Let us assume that the case $(III)$ occur, i.e., $P$ is contained in a union of a rational cubic $C_1$ and a degree two plane curve $C_2$. 
 In this case to  satisfy $\text{CB}(5)$, either $P \cap C_1$ has length at least $16$ or $P \cap C_2$ has length at least $11$. If $P \cap C_1$ has length $16$ and $P \cap C_2$ has length at most $10$, then $P$ fails to satisfy $\text{CB}(5)$. Similarly, if $P \cap C_2 $ has length $11$ and $P \cap C_1$ has length at most $15$, then $P$ fails to satisfy $\text{CB}(5)$. \\
 If $P \cap C_1$ has length $\ge 17$ and $P \cap C_2$ has length at most $8$, then $P$ fails to satisfy $\text{CB}(5)$ and if $P \cap C_2$ has length $= 9$, then one can find a quadratic hypersurface containing $C_1$ and a subscheme of length $2$ of $P \cap C_2$. Thus the residual has length $7$ and satisfies $\text{CB}(3)$ hence contained in a line, a contradiction. Similarly if $C_2 \cap P$ has length at least $12$ and $C_1 \cap P$ has length at most $14$, then one can get a contradiction.\
 
 Thus we have the following conclusion:
 If case $(II)$ or $(IV)$ occurs, then $P$ has length at least $30$. If case  $(III)$
 occurs, then $P$ has length at least $27$ and if case $(I)$ occurs, then either $P$ has length at least $27$ or $P$ has length $26$ and both lines occuring in case $(I)$ intersect the rational cubic curve in a subscheme of length $2$. 
 
 \end{remark}

 %(a)  there is a plane $H$ containing $5$ points of the residual subscheme such that no $4$ of them lie in a line then  they imposes independent conditions on quadrics (by Theorem \ref{T1}) on $H$, a contradiction.\\
%(b)  If there is  a plane $H$ containing at least $8$ points of the residual subscheme then also one can find five points in $H$, no four of them lie on a line, a contradiction.\\
  %$ Let $H$ be a plane containing at least $4$ points not satisfying the conditions on (a) and (b). If the residual sub-scheme with respect to $ X_1 H$ has length $\le 10$,  then  as
  % they satisfies CB(3), one can  easily see that they all lie on a plane. If the residual sub-scheme has length $\ge 11$, then again we will get a plane $G$ containing at least $4$ points. 
   %If $H$ and $G$ do not cover the residual sub-scheme then there exist at least three skew lines contained in $X_2$ and $X_2^\prime$, a contradiction.
   %Thus $X_1, G, H$ cover $Z$ and have the configuration as in Lemma \ref{NL4}, a contradiction.

    \section{Existence of non-trivial  sections of a general point of the moduli space for small second Chern class}\label{S3}

Let $S$ be a smooth irreducible projective surface over the field of complex numbers.
Let, $\mathcal{M}(r, H, c_2)$ be the moduli space of rank $r$, $H$-stable vector bundles with fixed determinant $H$ and second
Chern class $c_2$, where $H:= c_1(\mathcal{O}_S(1))$. Consider a point $E \in \mathcal{M}(r, H, c_2)$.
Then the obstruction theory is controlled by 
\[
 \text{Obs}(E) := H^2(S, \text{End}^0(E)).
\]
Here, $\text{End}^0(E):= \text{Ker}(\text{tr}: \text{End}(E) \longrightarrow \mathcal{O}_S).$\
\
\text{By Serre duality},
\[
 H^2(S, \text{End}^0(E)) \cong H^0(S, \text{End}^0(E) \otimes K_S)^*.
\]
So, $\text{Obs}(E) \ne \{0\}$ if and only if there exists a non-zero element $\varphi \in H^0(S, \text{End}^0(E) \otimes K_S)$.
In other words, there exists a twisted endomorphism 
\[
 \varphi: E \longrightarrow E \otimes K_S, \text{with } \text{tr}(\varphi) = 0. 
\]
A bundle $E$ is said to be obstructed if $\text{Obs}(E) \ne \{0\}$, otherwise we call $E$ to be unobstructed. 

Here we only consider the case when $r=2$. 

Let $E \in \mathcal{M}(2, H, c_2)$  be a point which fits into an exact sequence 
\begin{equation}\label{e1}
0 \to L \to E \to \mathcal{J}_P \otimes L^{\prime} \to 0,
\end{equation} 
where $P$ is a zero-dimensional subscheme of $S$ and  $\mathcal{J}_P$ denotes the ideal sheaf of $P$.
Then we have $\text{det}(E) \cong L \otimes L^{\prime}$, 
\[
E^* \otimes L^{\prime} \cong E \otimes L^*,
\]
\[
\text{End}^0(E) \otimes L \otimes L^{\prime} \cong \text{Sym}^2(E),
\]
and there is an exact sequence 
\begin{equation}\label{e2}
0 \to E \otimes (L^{\prime})^* \to \text{End}^0(E) \to \mathcal{J}_P^2 \otimes L^{\prime} \otimes L^* \to 0.
\end{equation}
 
From now on we specialize to the case when $S$ is a very general sextic surface in $\mathbb{P}^3$.  
 
Let  $S \subset \mathbb{P}^3$ be a  very general, smooth surface of degree $6$. Then the canonical line bundle $K_S \simeq \mathcal{O}_S(2)$.
We denote by $\mathcal{M}(H, c_2)$ the moduli space of rank 2, $  H- $ stable vector bundles $E$ on $S$,
with respect to $H := \mathcal{O}_S(1)$ and with $c_1(E)=H$ and $c_2(E) =c_2$.
 Note that in this case the stability and semistability are same. The expected dimension of $\mathcal{M}(H, c_2)$ is
 $4c_2 -c_1^2 - 3\chi(\mathcal{O}_S) = 4c_2 -39$. In this section we will show that a general member of each irreducible component of the moduli space admits a non-trivial section for low second Chern classes. More precisely we will prove the following Theorem,
 \begin{theorem}\label{TZ}
 If $5 \le c_2 \le 18$, then a general point $E$ of each irreducible component of   $ \mathcal{M}(H, c_2)$   admits a non-zero section. Further more a general element in each component of  the subspace $\{E \in \mathcal{M}(H, c_2): h^0(E)= 0\}$ is smooth. 
 \end{theorem}

%\begin{lemma}
%	Let $9 \le c_2 \le 16$. Then a general element $E$ in each compomemt of $\mathcal{M}(H, c_2)$ with $h^0(E) = 0$, corresponds to a reduced zero-dimensional subscheme of length $c_2 +12$.
%\end{lemma}
%\begin{proof}
%	Let $E$ be an element in an irreducible component of $\mathcal{M}(H, c_2)$ such that $h^0(E) = 0$. Then from the Euler characteristic  computation tells us that $E$ fits in an sequence of the form,
%	\[
%%%	where $Z$ is a locally complete intersection zero-dimensional subscheme of length $c_2 +12$.
	
	%\end{proof}
 \begin{proof}
 Note that, the Euler characteristic of a point $E \in \mathcal{M}(H, c_2)$ is equal to $19-c_2$ which is $ > 0$ for $c_2 \le 18$. Thus $h^0(S, E) + h^2(S,E) = h^0(S, E) + h^0(S, E(1)) > 0$ for  $c_2 \le 18$.  In particular, $H^0(S, E(1)) \ne 0$.    Let us assume $E$ does not have any section. Then we have the following
 exact sequence:
 \begin{equation}\label{e6}
 0 \to \mathcal{O}_S(-1) \to E \to \mathcal{J}_P(2) \to 0
 \end{equation} 
 where, $\mathcal{J}_P$ is the ideal sheaf of a zero-dimensional subscheme $P$ of length $12 + c_2$. Since $E$ is locally free, $P$ is locally complete intersection and satisfies $\text{CB}(5)$. \\
 %Note that the bundles occuring in the exact sequence can be represented by points in a subscheme $Y$of $\text{Hilb}^{c_2 +12}$ 
 %consisting of elements satisfying $\text{CB}(5)$. The open subset parametrizing the reduced subscheme of length $c_2 +12$ has a non-empty intersection with $Y$. Thus for a general element occuring in the above exact sequence, we can assume $P$ to be reduced. 
 Thus by Theorem \ref{TY},  we have the following possibilities :

$(I)$ $P$ is contained in a union of a rational cubic curve and a pair of skew lines.\\ 
$(II)$ $P$ is contained in a union of a rational cubic curve and a plane curve of degree $2$.\\
$(III)$ $P$ is contained in a union of a rational cubic curve, a line and a degree $2$ plane curve. \\
$(IV)$ $P$ is contained in a union of a rational cubic curve and a plane curve of degree $3$.\\

Let us assume the length of $P$ is different from $26$, which case we will consider separately. Then by Remark \ref{R11},
the first two cases occur when the length of $P \ge 27$ and last two cases occur when the length of $P \ge 30$. \\
Thus if $c_2 \le 13$, then every element $E \in \mathcal{M}(H, c_2)$ admits a section.\\
Let $15 \le c_2 \le 17$. Then  only the first two cases can occur.  Note that, $S$ intersects a rational cubic curve  in a subscheme of length $18$  and a degree $2$ curves or a pair of skew lines in a subscheme of length $12$.  Thus given a rational cubic curve and a pair of skew lines or a degree $2$ curve, there are only finitely many choices of $P$. On the other hand the dimension of rational cubic curves in $\mathbb{P}^3$ is $12$ and degree $2$ plane curves is $8$. Thus the dimension of such $P$'s is at most $20$, but the dimension of every irreducible component of $\mathcal{M}(H, c_2)$ is at least $4c_2 -39$ which is $> 20$ for $c_2 \ge 15$, which proves the claim in this range.\\
Let $c_2 = 30$. Then the dimension of such subschemes is at most $12 +12= 24$ and the dimension of every component of $\mathcal{M}(H, c_2)$ is at least $33$. 
Thus a general element of each of the irreducible components of $\mathcal{M}(H, c_2)$ admits a non-zero section.\\
Let us consider the case when $c_2= 14$. In that case, only first case can occur, that is, the subscheme $P$ is contained in a union of a rational cubic curve and a pair of skew lines each of which intersects the rational cubic curve in a subscheme of length $2$. Note that, dimension of lines intersecting a rational cubic curve in a subscheme of length $2$ is $2$, which can be seen in the following way:\\
Let $C_1$ be a rational cubic curve. \\
Set $Y:= \{(H, l) \text{ where } H \in \mathbb{P}(H^0(\mathcal{O}_S(1))), l \subset H, \text{ a line intersecting } C_1 \text{ in a subscheme of length $2$ } \}$\\
and let $W$ be the space of lines intersecting $C_1$ in a subscheme of length $2$. Then there is surjective projection map $Y \to W$. Note that, given a plane there are only finitely may such lines. Thus the dimension of $Y = 3$. On the other hand, given a point $l \in W$, there are one dimensional planes containing it.\\
Thus the dimension of $W =2$. Thus the dimension of pair of skew lines is at most $4$. Therefore, the dimension of such $P$'s is at most $16$, but the dimension of every irreducible component of $\mathcal{M}(H, 14)$ is at least $17$ which is $> 16$.\

%On the other hand,  taking $L =\mathcal{O}_S(-1)$ and $L^{\prime} = \mathcal{O}_S(2)$ in \ref{e2} we have:
 % \begin{equation}\label{e5}
 %0 \to E \otimes \mathcal{O}_S(-2) \to \text{End}^0(E) \to \mathcal{J}_P^2(3) \to 0.
 %\end{equation} 
 %Tensoring \ref{e5} by $\mathcal{O}_S(2)$ and considering the cohomology sequence, we have
 %\[
% h^0(S, \text{End}^0(E) \otimes \mathcal{O}_S(2)) \le h^0(S, E) + h^0(S, \mathcal{J}_P^2(5)) = h^0(S, \mathcal{J}_P^2(5)).
 %\] 
 %Thus the Theorem will follow if $h^0(S, \mathcal{J}_P^2(5)) = 0$.\\ 
% Since $h^0(E) = 0$, $P$ lies in one of the (I)-(V), configuration. 
 %Thus there is a rational cubic curve containing at least $12$ points of $P$. Now $h^0(S, \mathcal{J}_P^2(5)) \ne 0$ mean there is a quintic hypersurface section which is singular along $P$. Thus the quintic hypersurface section has to intersect the rational cubic curve at at least $24$ points. But $S$ intersects a rational cubic in $18$, a contradiction. 
 %Thus $h^0(S, \mathcal{J}_P^2(5)) = 0$ which conclude the theorem. 

      \end{proof}

 \section{Mestrano-Simpson conjecture}\label{S4}
    In this section we will prove the Mestrano-Simpson conjecture partially. Let us recall few results from \cite{SIM2}  and the Mestrano-Simpson conjecture.\\
     \begin{theorem}\label{t1}
     The space of bundles $E$ fitting into an exact sequence of the form
\[
 0 \to \mathcal{O}_S \to E \to \mathcal{J}_P(1) \to 0
 \]    
      where,  $P$ is a length $11$ subscheme of $C \cap S $ for $C$ a rational normal cubic in
$\mathbb{P}^3$, consists of a single $12-$dimensional generically smooth irreducible component
of the moduli space $\mathcal{M}(H, 11)$ of stable bundles 
on $S$.
 \end{theorem}  
 \begin{proof}
See  \cite[Corollary 11.3]{SIM2}.
\end{proof}
\begin{theorem}\label{t2}
 The space of bundles $E$ fitting into an exact sequence of the form
\[
 0 \to \mathcal{O}_S \to E \to \mathcal{J}_P(1) \to 0
 \]    
      where,  $P$ is a length $11$ subscheme of $H \cap S $ for $H$ a hyperplane, general with respect to $S$ in
$\mathbb{P}^3$, is contained in an  irreducible component of dimension $\ge 13$ of the moduli space $\mathcal{M}(H, 11)$ of stable bundles 
on $S$.
 \end{theorem}
 \begin{proof}
  See  \cite[Theorem 11.4]{SIM2}.
 \end{proof}
\begin{conjecture}\label{NC1}(Mestrano-Simpson Conjecture,  [\cite[Conjecture 11.5]{SIM2})\\
The $13-$dimensional family in Theorem \ref{t2}
constitutes a full irreducible component of $\mathcal{M}(H, 11)$ ;  this component is
nonreduced and obstructed. Together with the $12-$dimensional generically smooth
component  in  Theorem \ref{t1}, these are the only irreducible
components of $\mathcal{M}(H, 11)$. In particular, $h^0(E)> 0$ for any stable bundle $E \in \mathcal{M}(H, 11)$.
\end{conjecture}

By theorem \ref{TZ}
 any  point $E  \in \mathcal{M}(H, 11) $ sits in an exact sequence of the form:
\begin{equation}\label{NEA}
0 \to \mathcal{O} \to E \to  \mathcal{J}_P(1) \to 0,
\end{equation}
where $P$ is a locally complete intersection zero-dimensional subscheme of length $11$ satisfying $\text{CB}(3)$.\\
Let us consider the subsets
\[
\mathcal{M}_1:= \{E \in \mathcal{M}(H, 11) :  P \text{ in } \ref{NEA}  \text{ lies on a rational cubic curve } \}
\]
and 
\[
\mathcal{M}_2:= \{E \in \mathcal{M}(H, 11) :  P \text{ in } \ref{NEA}  \text{ lies on a plane } \}
\]
Clearly $\mathcal{M}_1$ and $\mathcal{M}_2$ are closed subset of $\mathcal{M}(H, 11)$ and by Theorem \ref{t1} $\mathcal{M}_1$ is an irreducible component of dimension $12$.

 \begin{proposition}\label{PP2}
 If there is no hyperplane passing through the zero-dimensional subscheme $P$ in the exact sequence \ref{NEA}, then $P$  lies in a rational normal cubic curve.
 \end{proposition}
 \begin{proof}
Since there is no hyperplane passing through $P$, from the cohomology sequence of \ref{NEA},  we have $h^0(E) =1$. On the other hand, the Euler characteristic computation says that 
$h^0(E(1)) \ge 7$. Again tensoring \ref{NEA} by $\mathcal{O}_S(1)$ and considering the cohomology sequence we have, $h^0(\mathcal{J}_P(2)) \ge 3$. In other words, $P$ lies in the intersection of at least $3$ quadrics.  If all the quadrics have a common component  $H$ where, $H$ is a hyperplane, then one can show that $P$ lies on $H$, a contradiction.
Choose two quadrics  $Q_1, Q_2$ with out having a common component and   passing through $P$.  Let $Q_3$ be another quadric which is not in the span of $Q_1, Q_2$. If $Q_3$ does not contain any component of $Y=: Q_1 \cap Q_2$, then  $Q_3$ intersects $Y$ in a subscheme of length $8$, a contradiction. Thus $Q_3$ contains a component of $Y$.\\
{\bf Case I}:   $Q_3$ contains a component, say, $C$ of degree $1$.\\
Then the remaining components of $Y$ intersect $Q_3$ in a subscheme of length $6$. Thus $C$ contains a subscheme of length at least $5$. Thus there always exists a plane $H$ containing a subscheme of length at least $6$ and the residual of $P$ with respect to $H$ satisfies $\text{CB}(2)$, hence the residual has length at least $4$ and at most $5$. Thus they all lie in a line, say $l$. Clearly $l$ can not intersect $C$, otherwise we will get a plane containing a subscheme of length $9$ hence all of $P$. Thus $C$ and $l$ are skew lines. If $C$ contains a subscheme of length exactly $5$  of $P$, then a subscheme of length $10$  of $P$ lie in a pair of skew lines and a subscheme of length $1$ out side the union of $C$ and $l$. In such a situation one can see that $P$ fails to satisfy $\text{CB}(3)$, a contradiction. \\
If $C$ contains a subscheme of length $6$, then $l$ can contains a subscheme of length  exactly $4$, again one can see that $P$ can not satisfy $\text{CB}(3)$.  \\
{\bf Case II}:  $Q_3$ contains a component $C$ of degree $2$.\\
In this case the remaining components of $Y$ intersect $Q_3$ in a subscheme of length $4$. Thus $C$ contains a subscheme of length at least $7$. 
Since any degree $2$ space curve  lie in a plane,  there is a plane containing a subscheme of length $7$ of $P$. In this case arguing as case I, one can get a contradiction.\\
Thus, the only possibility is $Q_3$ contains a component $C$ of degree $3$. Now any degree three space curve is either a plane curve or it is a rational normal cubic curve.
In case of plane curve we can easily get a contradiction. \\ 
Thus  $C$ is a rational normal cubic which contains a subscheme of length at least $9$ of $P$. If the remaining points of $P$ do not lie on $C$, then $P$ imposes independent conditions on sections of $\mathcal{O}(3)$, hence fails to satisfy $\text{CB}(3)$, a contradiction. Therefore, $P$ lies on $C$, which concludes the proposition.

 \end{proof}   
 
 \begin{proposition}\label{NP1}
 The subset $\mathcal{M}_2$ is of $13$ dimensional and it can have at most $2$ irreducible components, one of  dimension $13$ and  possibly one of dimension $12$. 
 \end{proposition}
 \begin{proof}
 By Theorem \ref{t2}, $\mathcal{M}_2$ has dimension at least $13$.  Note that $\mathcal{M}_2$ can be identified as a subspace  of the space of triples,
 \[
 \{(P, H, E) \in \text{Hilb}^{11}(S) \times \mathbb{P}(H^0(\mathcal{O}(1))) \times  \cup_{Z \in \text{Hilb}^{11}(S)}\mathbb{P}(H^1(\mathcal{J}_Z(3))): P \subset H, E \in \mathbb{P}(H^1(\mathcal{J}_P(3)))\}
 \]
 If $H$ is in general position with respect to $S$ and $P$ is a general subscheme of length $11$, then one can show (as in \cite{SIM2}), that such collection of triples constitute 
 a $13$ dimensional irreducible subset $\mathcal{M}_2^1$. 
 
 Note that, if the minimal degree curve containing $P$ has degree $\ge 4$, then one can show that $P$ fails to satisfy $\text{CB}(3)$ ( as a general cubic hypersurface intersects a degree $4$ curve in a subscheme of length $12$ where as the length of $P$ is $11$).\\
 If $P$ is contained in a plane curve of dgree $2$, then one can show that the dimension of such  subschemes is strictly smaller than $12$ but by [\cite{SIM1}, Corollary 3.1],  every component of the space of bundles  $E$  which fits in the exact sequence of the form
  \[
   0 \to \mathcal{O}_S \to E \to \mathcal{J}_P(1) \to 0
   \]
   has dimension at least $12$.\\
 Consider the subset 
$ \mathcal{M}_2^2:= \{(P, H, E) \in \mathcal{M}_2\}$ such that $P$ is contained in an irreducible cubic curve in $H$. Then since an irreducible plane cubic curve intersects another cubic 
curve at $9$ points, $h^0(\mathcal{J}_P(3)) = 11$ and such $P$ satisfies $\text{CB}(3)$. Thus the space of extension $\text{Ext}^1(\mathcal{J}_P(1), \mathcal{O}_S)$ has dimension $2$. Also 
note that, dimension of such sub-schemes in a plane irreducible cubic curve $C$ is zero ( there are only finitely many choices in  $ S \cap C$). Using monodromy argument one can 
show that $\mathcal{M}_2^2$ is an irreducible subset of $\mathcal{M}_2$

On the other hand,
we have $h^0(\mathcal{J}_P(1))= 1$, since P is contained in a plane, so $h^0(E)= 2$. This
means that for a given bundle $E$, the space of choices of sections (modulo scaling)
leading to the subscheme $P$, has dimension $1$. Hence the dimension of the space of
bundles obtained by this construction is one less than the dimension of the space of
subschemes.

Count now the dimension of the space of choices of $P$: there is a three-dimensional
space of choices of the plane $H$, and for each one we have an $9$-dimensional space
of choices of irreducible cubic curves in the plane $H$ and for each choice of $P$, one dimensional space of extensions upto scalars.  This gives $\text{dim}\{P\}= 3 + 9 + 1 -1=12$.

 Let us consider the other subsets, that is the  sub-schemes $P$ such that  the dimension of the space of extensions
 $H^1(\mathcal{J}_P(3))$ is high.
 So let us consider the set $\triangle_{H, i}:= \{ P \in Hilb^{11}(S \cap H) : h^0(H, \mathcal{J}_{P, H}(3)) = 1 +i\}$. Note that, for $P \in \triangle_{H, i}, h^1(\mathcal{J}_P(3)) = i +2$.
 Consider the incidence variety $T = \{(Z, C): Z \subset C\} \subset \mathbb{P}(H^0(\mathcal{O}_H(3))) \times \text{Hilb}^{11}(S \cap H)$
and let $\pi_1, \pi_2$ be the projections. Since $Z$ is contained in at least two cubic plane curve, $C$ has to be reducible. In other words, the image of $\pi_2$ has dimension 
at most $8$.\\
Note that, 
the dimension
of
$\pi_2^{-1}(C)$ is  $0$.\\
%{\bf proof of the claim}\footnote{ The proof of the claim is taken from Mathoverflow: ogijj (https://mathoverflow.net/users/98919/ogijj), complete intersection curves with large Hilbert scheme of points, URL (version: 2020-02-14): https://mathoverflow.net/q/352716 }:\\
%Note that that $\pi_2^{-1}(C) = \text{Hilb}^{11}(S \cap C)$.\\
%Let $Y = S \cap C$. For a point $p\in Y$ and any $e\in \mathbb{N}$ denote by $\text{Hilb}^e(Y, p) \subset \text{Hilb}^e(Y)$ the locus consisting of subschemes supported only at $p$. Since $S \subset \mathbb{P}^3$ is smooth, the singularities of $Y$ are planar, so that for every $p\in Y$ the locus $\text{Hilb}^e(Y, p)$ can be identified with a subscheme of $\text{Hilb}^e(\mathbb{A}^2, 0)$.
%By \cite{Br}  $\dim \text{Hilb}^e(\mathbb{A}^2, 0) = e-1$, so $\text{Hilb}^e(Y, p)\leq e-1$. (here we need characteristic zero)

%Now fix a partition $\lambda = (\lambda_1,\ldots, \lambda_m)$ of $11$ and consider the locus $S_{\lambda}$ in $\text{Sym}^{11}(Y)$ consisting of cycles of the form $\sum \lambda_i p_i$ for distinct $p_i\in Y$. By the above estimate, the fiber of Hilbert-Chow morphism over any element of $S_{\lambda}$ has dimension at most $11-m$. The locus $S_{\lambda}$ has codimension $11-m$ in $\text{Sym}^d(Y)$, so we get that its pre-image has dimension at most $11$, summing over all $\lambda$ we get the claim.
Thus, we have, $8 \ge \text{dim }T \ge \text{dim }\pi_1^{-1}(\triangle_{H, i}) \ge \text{dim }(\triangle_{H, i})  + i $.
This implies that $\text{dim }\triangle_{H, i}$ is bounded by $8-i$. Hence  as earlier the space of bundles $E$ obtained by this construction has dimension 
$\le \text{dim}(\cup_H(\triangle_{H, i}) +$ the dimension of the extensions $-1 = 3 + 8 -i + i+1 -1= 11$. But by [\cite{SIM1}, Corollary 3.1],  every component of the space of bundles  $E$  which fits in the exact sequence of the form
  \[
   0 \to \mathcal{O}_S \to E \to \mathcal{J}_P(1) \to 0
   \]
   has dimension at least $12$ and closure of $\mathcal{M}_2^1 \cup \mathcal{M}_2^2$ covers $\mathcal{M}_2$. Hence $\mathcal{M}_2$ can have at most two irreducible components
   of dimension $13$ and $12$ respectively.
 %co-dimension of $\triangle_i$ in
%$\text{Hilb}^{c_2}(S)$ is $\ge i + 1$. 

  \end{proof}
 
 Thus we have the following theorem.
\begin{theorem}
The moduli space $\mathcal{M}(H, 11)$ is covered by $\mathcal{M}_1$ and $\mathcal{M}_2$. In particular, $h^0(E) \ne 0$ for all $E \in \mathcal{M}(H, 11)$ and all the irreducible component of $\mathcal{M}(H, 11)$ has dimension at most $13$. Further more $\mathcal{M}_2$ contains at most $2$ irreducible components one of dimension $13$ and possibly
another of dimension $12$. 
\end{theorem}
\begin{remark}
If $\overline{\mathcal{M}_2^1}$ contains $\mathcal{M}_2^2$, then $\mathcal{M}_2$ is irreducible and hence  the the conjecture \ref{NC1} is true.
\end{remark}

 \section{Dimension estimates of $\mathcal{M}(H, c_2)$ and its boundary for $c_2 \le 19$}\label{S5}
 It is known \cite[Corollary 17.10]{SIM3} that the moduli space $\mathcal{M}(H, c_2)$ is empty for $c_2 \le 4$.
   \begin{proposition}
 For $c_2= 5, 6,   \mathcal{M}(H, c_2)$ is irreducible of dimension $2 $ and $3$ respectively and $\mathcal{M}(H, 7)$ is empty.
 \end{proposition}
 \begin{proof}
    Let $E \in \mathcal{M}(H, c_2), c_2= 5, 6, 7$.  Then by theorem \ref{TZ},we have the following exact sequence
  \[
  0 \to \mathcal{O}_S \to E \to \mathcal{J}_P(1) \to 0
  \]   
  where, $P$ is a zero-dimensional locally complete intersection subscheme of $S$ of length $c_2$  and satisfies $\text{CB}(3)$. Thus by \cite[Proposition 17.8]{SIM3}, $P$ lies in a line $l$.  Since a line intersects
  $S$ in a subscheme of length $6$, $\mathcal{M}(H, 7)$ is empty.

   Note that, a linear form which vanishes on $P$, also vanishes on $l$. Thus
  $h^0(S, \mathcal{J}_P(1)) = h^0(S, \mathcal{J}_l(1)) = 2$. Therefore, $h^0(E) = 3$.\\ 
  {\bf Case I}: $c_2 = 5$.\\
   Let $R$ be the subsheaf of $E$ generated by its global sections and $T$ be the co-kernel in the exact sequence
   \[
   0 \to R \to E \to T \to 0.
   \]
   We also have an exact sequence
   \[
   0 \to \mathcal{J}_{l \cap S}(1) \to \mathcal{J}_P(1) \to T \to 0.
   \]
   So $T$ has length $1$. It is supported on a point $p \in S$.
    Since  $h^0(E) = 3, R$ is generated by $3$ sections. Thus we have the following exact sequence 
   \[
   0 \to \text{ker} \to \mathcal{O}_S^3 \to R \to 0
   \]
   with locally free kernel. From the degree computation, we have $\text{ker} = \mathcal{O}_S(-1)$.
   
   On the other hand, if $p \in \mathbb{P}^3$ is a point and $G:= H^0(\mathbb{P}^3, \mathcal{J}_p(1))$, then we have the canonical exact sequence
   \[
   0 \to \mathcal{O}(-1) \to G^* \otimes \mathcal{O} \to \mathcal{R}_p \to 0
   \] 
   where, $\mathcal{R}_p$ is a reflexive sheaf with $c_2(\mathcal{R}_p) =\text{ the class of the line }$. Thus ${\mathcal{R}_p}{\mid_S}$ has $c_2 = 6$.
   
   Thus we can conclude that $R = {\mathcal{R}_p}{\mid_S}$ and $E \cong {{\mathcal{R}_p}{\mid_S}}^{**}$. 
   Thus we get a map $\mathcal{M}(H, 5) \to S$ which takes $E \to p$, the support of $T$,  with the inverse map, $p \to {{\mathcal{R}_p}{\mid_S}}^{**}$, which gives an isomorphism 
   of $\mathcal{M}(H, 5)$ to $S$.
   
   {\bf Case II}: $c_2= 6$.\\
   In this case $P = l \cap S$. In other words, $E$ itself is  generated by its sections. 
   Thus we have an exact sequence 
   \[
   0 \to \mathcal{O}(-1) \to \mathcal{O}_S^3 \to E \to 0,
   \]
   which gives a subspace of dimension $3$ of the space of linear forms in $\mathbb{P}^3$. Let $p$ be the base locus of this $3-$dimensional subspace.  Then it is easy to see that   
   $E$ is  isomorphic to $\mathcal{R}_p$  which defines a map  
   $f: \mathcal{M}(H, 6) \to \mathbb{P}^3$. 
   On the other hand,  the restriction of the co-kernel sheaf $\mathcal{R}_p$ to $S$ is locally free if and only if $p \in \mathbb{P}^3 \setminus S$. Thus image of $f$ is contained in 
   $\mathbb{P}^3 \setminus S$. In other words, we have a map $f: \mathcal{M}(H, 6) \to \mathbb{P}^3 \setminus S$ with  inverse $p \to {\mathcal{R}_p}{\mid_S}$, which proves the 
   Proposition.
   \end{proof}
 \subsection{When $8 \le c_2 \le 10$}

 Let $E \in \mathcal{M}(H, c_2), 8 \le c_2 \le 10$ .   Then $E$ fits in the following exact sequence:
 \begin{equation}\label{EE1}
 0 \to \mathcal{O}_S \to E \to \mathcal{J}_P(1) \to 0
 \end{equation}
 where, $P$ is a zero-dimensional subscheme of $S$ of length $c_2$ and  $P$ satisfies $\text{CB}(3)$. Thus by  \cite[Proposition 17.8]{SIM3}, $P$ lies on 
 a hyperplane or on a pair of skew lines or on a double line.
 \begin{proposition} 
 $\mathcal{M}(H, 8)$ is irreducible of dimension $7$. 
 \end{proposition}
 \begin{proof}
  
 %Since the expected dimension for $c_2= 8$ is negative, by Theorem,\ref{T2}, every bundle $E \in \mathcal{M}(H, 8)$ admits a section. 
  %Thus by [\cite{SIM3}, Proposition 17.8] , $P$ in \ref{EE1}   lies on a hyperplane $H$.\\
 {\bf Claim}:  $P$ in \ref{EE1} lies in a plane curve of degree $2$. \\
 {\bf Proof of the claim}: \\
  By \cite[Proposition 17.8]{SIM3} , $P$ in \ref{EE1}   lies on a hyperplane $H$.  Thus $h^0(\mathcal{J}_P(1)) =1$. Note that, $h^2(\mathcal{J}_P(1)) = h^2(\mathcal{O}_S(1)) = 4$. Thus  from the long exact sequence of the canonical sequence
  \[
  0 \to \mathcal{J}_P(1) \to \mathcal{O}_S(1) \to \mathcal{O}_P(1) \to 0,
  \]
  we have $h^1(\mathcal{J}_P(1)) = 5$.
  On the other hand, from the long exact sequence of \ref{EE1}, we have the following exact sequence:
  \[
  0 \to H^1(E) \to H^1(\mathcal{J}_P(1)) \to H^2(\mathcal{O}_S) \to H^2(E) \to H^2(\mathcal{J}_P(1)) \to 0.
  \]
  Thus we have $h^0(E(1)) = h^2(E) \ge 9$. Tensoring \ref{EE1} by $\mathcal{O}_S(1)$ and considering its long exact sequence, one can see that $h^0(\mathcal{J}_P(2)) \ge 5$.
  Therefore, $P$ lies in at least $5$ quadratic hypersurface sections. But the quadratic hypersurface sections containing $H$ has dimension $4$. Thus there is a quadratic hypersurface 
  $Q$ containing $P$ but not containing $H$. Therefore, $Q \cap H$ gives a plane curve of degree $2$ containing $P$, which proves our claim.
  
  Since any two plane curves of degree $2$ intersect at a subscheme of length atmost $4$  of $S$, there exist a unique plane curve of degree $2$ containing $P$. On the other hand, given a plane curve of degree $2$, it intersects $S$ in a subscheme of length $12$  and any subscheme of length $8$  of this subscheme of length $12$  satisfies $\text{CB}(3)$, hence gives a vector bundle as an
  extension of the form:
  \[
  0 \to \mathcal{O} \to E \to \mathcal{J}_P(1) \to 0.
  \]
  Note that, such an extension is unique up to isomorphism.  Therefore,  $\mathcal{M}(H, 8)$ is isomorphic to the space of bundles $E$ fitting into an exact sequence of the form
\ref{EE1} , where $P$ is a length $8$ subscheme of $C \cap S$ for $C$ a rational plane curve of degree $2$ in $\mathbb{P}^3$. Now using the monodromy  argument on the 
  set of choices of subschemes of length $8$  out of subscheme of length $12$  of $C \cap S$, as $C$ moves one can show that $\mathcal{M}(H, 8)$ is irreducible.  Now the dimension of plane curves in 
  $\mathbb{P}^3$ is $8$ and  since $h^0(E) = 2$, the dimension of $\mathcal{M}(H, 8)$ is $8-1=7$, which concludes the Proposition.  
  \end{proof}

 \begin{lemma}\label{P1}
 Let $ 9 \le c_2 \le 10$. Then there is a bijection between the irreducible components of $\mathcal{M}(H, c_2)$ and the irreducible components of the space of bundles $E$ fitting into the exact sequence of the form \ref{EE1}, where $P$ is a length $c_2$ subscheme lying in a hyperplane.
   \end{lemma}
\begin{proof}
 
  {\bf Case I}: $c_2= 9$.\\
 In this case $P$ lies on a hyperplane or on a pair of skew lines each of which contains a subscheme of length at least $4$  of $P$ or $P$ lies on a double line.\\
 If $P$ lies on a pair of skew lines say, $l_1$ and $l_2$. Then one of them say, $l_1$ contains a subscheme of length exactly $4$  of $P$ and imposes $4$ independent conditions on sections 
 of $\mathcal{O}(3)$, which implies that $P$ can not satisfy $\text{CB}(3)$. Therefore, $P$ can not lie on a pair of skew lines.\\
   Since the dimension of lines in $\mathbb{P}^3$ is $4$, it is easy to see that the space of bundles $E$ fitting into an exact sequence of the form
 \ref{EE1},  where $P$ is a length $9$ subscheme lying in a double line has dimension strictly smaller than $6$. On the other hand, by \cite[Corollary 3.1]{SIM1}, 
 every irreducible component of $\mathcal{M}(H, 9)$ has dimension at least $6$. Therefore there is a bijection between  the irreducible components 
 of $\mathcal{M}(H, 9)$ and the irreducible components of the space of bundles $E$ fitting into an exact sequence of the form \ref{EE1}, where $P$ is a length $9$ subscheme 
 lying on a hyperplane. \\
 
 {\bf Case II}: $c_2 = 10$. \\
  By [\cite{SIM1}, Corollary 3.1], 
every irreducible component of $\mathcal{M}(H, 10)$ has dimension at least $9$.  On the other hand, if $P$ lies on a pair of skew lines or on a double line, then one can easily see
 that the dimension of the space of bundles fitting into an exact sequence of the form
 \ref{EE1},  where $P$ is a length $10$ subscheme lying in a pair of skew lines or in a double line has dimension strictly smaller than $9$. Thus as in  earlier case 
there is a bijection between the  the irreducible components 
 of $\mathcal{M}(H, 10)$ and the irreducible components of the space of bundles $E$ fitting into an exact sequence of the form \ref{EE1}, where $P$ is a length $10$ subscheme 
 lying on a hyperplane. \\
 \end{proof}
 \begin{proposition}
 $\mathcal{M}(H, 9)$ is irreducible of dimension $\le 10$. 
  \end{proposition}
 \begin{proof} 
 By Lemma \ref{P1},  it is enough to consider the space of bundles $E$ fitting into an exact sequence of the form \ref{EE1}, where $P$ is a length $ 9$ subscheme 
 lying on a hyperplane.\\
 Let $H$ be a plane in general position with respect to
$S$, and let $Y: = S  \cap  H$. Let $P$ consists of a general subscheme of length $9$  in $Y$.
 The map $H^0(\mathcal{O}_H(3)) \to  H^0(\mathcal{O}_Y(3))$
is injective (since $Y$ is a curve of degree $6$ in the plane $H$), so a general subscheme
of length $9$  in $Y$ imposes independent conditions on $H^0(\mathcal{O}_H(3))$.  Thus a general subscheme of length $9$   fails to satisfy $\text{CB}(3)$.
Let $Z \subset \text{Hilb}^9(Y)$ be the subscheme  which fails to impose independent conditions on sections of $\mathcal{O}_H(3)$.  Note that, then such a subscheme $P$ is contained  in a complete intersection of two cubics.
On the other hand, if a subscheme $P$ is contained in a complete itersection of two cubics, then  by Remark \ref{R}, $P$ satisfies $\text{CB}(3)$ and hence gives a point in $\mathcal{M}(H, 9)$. Now the space of  complete intersections of two cubics is irreducible. Again using monodromy argument one can show that $\mathcal{M}(H, 9)$ is irreducible.\\ 
Let us compute the dimension of such subschemes.
Note that, 
 $\text{dim}(Z) \le 8$. Also for a general point 
$P \in Z, h^0(\mathcal{J}_{P/H}(3)) = 2$, where $\mathcal{J}_{P/H}$ denotes the ideal sheaf of $P$ in $H$. Thus $h^0(\mathcal{J}_P(3)) = 12$. In other words, $h^1(\mathcal{J}_P(3)) =1$. Thus $P$ determines a vector bundle $E$ uniquely up to isomorphism.  
Since $h^0(\mathcal{J}_P(1)) = 1$, we have $h^0(E) = 2$. This
means that for a given bundle $E$, the space of choices of sections $s$ (modulo scaling)
leading to the subscheme $P$, has dimension $1$. Hence the dimension of the space of
bundles obtained by this construction is one less than the dimension of the space of
subschemes.\\
 Now there is a three-dimensional
space of choices of the plane $H$, and for each one we have an $9-$dimensional space
of choices of the subscheme $P$.  This gives the total dimension of the space of choices of $P = 3 + \text{dim}(Z)-1$. So total dimension of such bundles is 
 $3 + \text{dim}(Z) -1$ which is $\le 10$.\\
 
 \end{proof}
 
\begin{proposition}
$\mathcal{M}(H,10)$ is irreducible of dimension $11$.
\end{proposition}
\begin{proof} 
By Lemma \ref{P1},  it is enough to consider the space of bundles $E$ fitting into an exact sequence of the form \ref{EE1}, where $P$ is a length $ 10$ subscheme 
 lying on a hyperplane.\\

  Since  $P$ satisfies $\text{CB}(3)$, $h^0(\mathcal{J}_P(3)) \ge 11$. Thus there exists a cubic 
 hypersurface containing $P$ and not containing $H$. By Bertini, a general such cubic is irreducible and hence intersects $H$ in an irreducible cubic plane curve.  Therefore, $P$
 lies in a plane cubic curve. 
 Since any two plane  cubic curves intersect in a subscheme of length   $9$ ,  there exists a unique plane cubic curve  containing $P$.\\
  On the other hand, given a plane cubic curve 
  intersects  $S$ in a subscheme of length $18$   and a plane cubic curve imposes $9$ independent conditions on the sections of $\mathcal{O}(3)$. Thus  any subscheme of length $10$  of this subscheme of length $18$   
   satisfies $\text{CB}(3)$ and  hence gives a vector bundle as an
  extension of the form:
  \[
  0 \to \mathcal{O} \to E \to \mathcal{J}_P(1) \to 0.
  \]
  Note that, such an extension is unique up to isomorphism.  Therefore,  $\mathcal{M}(H, 10)$ is isomorphic to the space of bundles $E$ fitting into an exact sequence of the form
\ref{EE1}  where, $P$ is a length $10$ subscheme of $C \cap S$ for $C$ a  plane cubic curve  in $\mathbb{P}^3$. Now using the monodromy  argument on the 
  set of choices of subschemes of length $10$  out of subscheme of length $18$  of $C \cap S$, as $C$ moves one can show that $\mathcal{M}(H, 10)$ is irreducible. \\
      Let $Y$ be a plane cubic curve.  Then for a general point $P$ in $\text{Hilb}^{c_2}(Y)$
one has $ h^0(\mathcal{J}_P(3)) = 21-c_2$, which gives $h^1(\mathcal{J}_P(3)) =1$. In other words, $P$ determines a vector bundle $E$ uniquely up to isomorphism.
 Since $h^0(\mathcal{J}_P(1)) = 1$, we have $h^0(E) = 2$. This
means that for a given bundle $E$, the space of choices of section $s$ (modulo scaling)
leading to the subscheme $P$, has dimension $1$. 
Hence the dimension of the space of
bundles obtained by this construction is one less than the dimension of the space of
subschemes.\\
 Now there is a three-dimensional
space of choices of the plane $H$, and for each one we have a $10-$dimensional space
of choices of the cubic curves.  This gives the total dimension of the space of choices of $P = 3 + 9=12$. So total dimension of such bundles is 
 $12 -1 =11$.
   \end{proof}
   \begin{remark}
   Note that $\mathcal{M}(H, 11)$ has two components of dimension $12$ and $13$, respectively.
   \end{remark}
  \subsection{$c_2 \ge 12$}
  Let $E \in \mathcal{M}(H, c_2)$ be a general point of an irreducible component of the moduli space.
  Then we have 
   \begin{equation}
 0 \to \mathcal{O}_S \to E \to \mathcal{J}_P(1) \to 0
 \end{equation}
 where, $P$ is a zero dimensional subscheme of $S$ of length $c_2$. \\ 
 Let $P$ lies in a hyperplane $H$ and $P^{\prime} \subset P$ be a subscheme of length $10$. Then $h^0(H,\mathcal{J}_{P^{\prime}/H}(3)) = 0$ and $h^0(E) =2$. 
 So any cubic hypersurface which contains $P$ will also contain $H$. Therefore,  $h^0(\mathcal{J}_P(3)) = 10$ and hence $h^1(\mathcal{J}_P(3)) = c_2-10$. 
 The total dimension of bundles which fit in the above exact sequence with $P$ lying in a plane is equal to the dimension of such subschemes $-1 + c_2 -11$.
 Now the dimension of such subschemes is $3 + c_2$ Thus Total dimension is $2c_2 -9$. 
 
  Also taking $L =\mathcal{O}_S$ and $L^{\prime} = \mathcal{O}_S(1)$ in \ref{e2} we have:
  \begin{equation}
 0 \to E \otimes \mathcal{O}_S(-1) \to \text{End}^0(E) \to \mathcal{J}_P^2(1) \to 0.
 \end{equation} 
 Tensoring the above exact sequence  by $\mathcal{O}_S(2)$ and considering the cohomology sequence, we have
 \[
 h^0(S, \text{End}^0(E) \otimes \mathcal{O}_S(2)) \le h^0(S, E(1)) + h^0(S, \mathcal{J}_P^2(3)).
 \] 
 If $P$ is not in any hyperplane section, then $h^0(\mathcal{J}_P(1)) = 0$. If $h^0(S, \mathcal{J}_P^2(3)) \ne 0$, then there is a cubic hypersurface  which is singular along $P$. But
 an irreducible cubic hypersurface can have at most $4$ isolated singularities \cite{BL}, thus any non-zero section of $\mathcal{J}_P^2(3)$ is reducible, in this situation it has two 
 components, a hyperplane and a quadratic hypersurface. On the other hand,
 any  irreducible quadratic hypersurface has only one isolated singularity \cite{BL}. Therefore,  $h^0(S, \mathcal{J}_P^2(3)) = 0$. 
 Thus we have  $ h^0(S, \text{End}^0(E) \otimes \mathcal{O}_S(2)) \le h^0(S, E(1))= h^0(\mathcal{O}(1)) + h^0(\mathcal{J}_P(2))= 4 + h^0(\mathcal{J}_P(2))$. \\
 
% Since any $3$ points in $P$ impose independent conditions on sections of $\mathcal{O}_S(2), h^0(\mathcal{J}_P(2)) \le 7$. If $h^0(\mathcal{J}_P(2)) =7$, then any $4$ points in $P$ satisfies CB(2), hence lie on a line, from which one can see that $h^0(\mathcal{J}_P(2)) \le 6$.\\
    If $h^0(\mathcal{J}_P(2)) \ge 3$, then $P$ lies in a  complete intersection of two quadratic hypersurfaces with out having any common component , i.e. in a curve $C$ of degree $4$. Thus for $c_2 \ge 13$, any cubic hypersurface
 which contains $P$ also contains $C$.  Now a complete intersection curve of degree $4$ 
 imposes $12$ independent conditions on sections of $\mathcal{O}(3)$. Thus we have, $h^0(\mathcal{J}_P(3)) = 8$. Therefore, the dimension of the isomorphic classes of bundles determined by $P, $ is $h^1(\mathcal{J}_P(3)) -1= c_2 -11$. Also we have $P$ lies  in a curve $Y = Q \cap S$, where $Q$ is a quadratic hypersurface and $Q$ varies
 over a $9$ dimensional variety of quadrics. Thus the dimension of such zero
 dimensional subscemes of length is at most $c_2 + 9$. Thus the total dimension of such bundles is at most $c_2 + 9 + c_2 -11 = 2c_2-2$. \\
 On the other hand, if $h^0(\mathcal{J}_P(2)) \le 2$, then $h^0(S, \text{End}^0(E) \otimes \mathcal{O}_S(2)) \le 6$. Thus in this situation,  every component of $\mathcal M(H, c_2)$ has dimension at most $4c_2 -39 + 6= 4c_2-33$(the expected dimension + the dimension of obstructions).
 Thus we have the following Proposition.
 \begin{proposition}
 Let $c_2 \ge 13$. Then every component of $\mathcal M(H, c_2)$ has dimension at most $\text{max}\{2c_2-2, 4c_2-33\}$ and for $c_2 = 12$, it has dimension $4c_2-39 + 10= 4c_2-29$.
 \end{proposition}
 
 \section{Boundary strata}\label{S6}
The boundary $\partial{\mathcal{M}(H, c_2)}: = \overline{\mathcal{M}}(H, c_2) - \mathcal{M}(H, c_2)$ is the set of points 
corresponding to torsion-free sheaves which are not locally free.

Let $\mathcal{M}(c_2, c_2^\prime):= \{[F] \in \overline{\mathcal{M}}(H, c_2)| F \text{ is not locally free with } c_2(F^{**})=c_2^\prime \}$.
Then the boundary has a decomposition into locally closed subsets 
\[
\partial{\mathcal{M}(H, c_2)}= \amalg_{c_2^\prime < c_2} \mathcal{M}(c_2, c_2^\prime).
\]
By the construction of $\mathcal{M}(c_2, c_2^\prime)$, we have a well defined map,
\[
 \mathcal{M}(c_2, c_2^\prime) \longrightarrow \mathcal{M}(H, c_2^\prime).
\]
The map  takes $E \longrightarrow E^{**}$. The fiber over $E \in \mathcal{M}(H, c_2^\prime)$ is the Grothendieck Quot-scheme
$\text{Quot}(E; d)$ of quotients of $E$ of length $d: = c_2 -c_2^\prime$.  Thus 
$\text{dim}(\mathcal{M}(c_2, c_2^\prime)) = \text{dim}(\mathcal{M}(H, c_2^\prime)) + \text{dim}(\text{Quot}(E; d))$. Now the 
dimension of $\text{Quot}(E; d)$ is $3(c_2-c_2^\prime)$. 
Therefore, 
\begin{equation}\label{EQX}
 \text{dim}(\mathcal{M}(c_2, c_2^\prime)) = \text{dim}(\mathcal{M}(H, c_2^\prime)) + 3(c_2-c_2^\prime).
\end{equation}
\ref{EQX} allows us to fill in the dimensions of the strata $\mathcal{M}(c_2, c_2^{\prime})$ in the Tables \ref{table:1} and \ref{table:2} 
starting from the dimensions of the moduli spaces given by previous section.
 The entries in the second column are the expected dimensions $ 4c_2-39$; in
the third column the upper bounds of the  dimensions of $\mathcal{M}(H, c_2)$ ; and in the following columns, the upper bounds of the dimensions of 
 $\mathcal{M}(c_2, c_2-d), d=, 1, 2, ..., 19$. \\

 \newpage
\begin{table}
\begin{center}
\begin{tabular}{ |c|c|c|c|c|c|c|c|c|c|c|c|} 
 \hline
 $c_2$ & e.d & $dim(\mathcal{M}) \le$ & $d=1$ & $d=2$ & $d=3 $  & $d=4$ & $d=5$ & $d=6$ & $d=7$ & $d=8$ & $d=9$ \\ 
5 & -19 &  2 & -- & --& --&-- &-- & --& --& --& --\\ 
 6 & -15 &  3 & 5 & -- & -- & -- & -- & -- & -- & -- & --\\ 
 7 & -11 & -1 & -- & -- & -- & -- & -- & -- & -- & -- & -- \\
 8 & -7  &  7 & -- & 9 & 11 & -- & -- & -- & -- & -- & --  \\
 9 & -3 & 10 & 10 &-- & 12 & 14 & -- & -- & -- & -- & -- \\
 10 & 1 & 11 & 13 & 13 & --& 15 & 17 & -- & -- & -- & -- \\
  11 & 5 & 13 & 14 & 16 & 16 & -- & 18 & 20 & -- & -- & -- \\
  12 & 9 & 19 & 16 & 17 & 19 & 19 & -- & 21& 23& -- & --\\
  13 & 13& 24 & 22 & 19 & 20 & 22& 22& --& 24 & 26& --\\
  14 & 17 & 26 & 27 & 25 & 22 & 23 & 25 & 25& --& 27 & 29\\
  15 & 21 & 28 & 29& 30 & 28 & 25 & 26 & 28 & 28 & --& 30 \\
  16 & 25 & 30 & 31 & 32 & 33 & 31 & 28 & 28& 31 & 31 & --\\
  17 & 29 & 34 & 33 & 34 & 35 & 36 & 34 & 31 & 31 & 34 & 34\\
  18 & 33 & 38 & 37 & 36 & 37 & 38 & 39 & 37 & 34 & 34 & 37\\
  19 & 37 & 42 & 41 & 40 & 39 & 40 & 41 & 42 & 40 & 37 & 37 \\
  
 \hline
\end{tabular}

\caption{upper bounds of dimensions of strata}
\label{table:1}
\end{center}

\end{table}

\begin{table}
\begin{center}
\begin{tabular}{ |c|c|c|c|c|c|c|c|} 
 \hline
 $c_2$ & e.d & $dim(\mathcal{M}) \le$ & $d=10$ & $d=11$ & $d=12 $  & $d=13$ & $d=14$  \\ 
5 & -19 &  2 & -- & --& --&-- &-- \\ 
 6 & -15 &  3 & -- & -- & -- & -- & -- \\ 
 7 & -11 & -1 & -- & -- & -- & -- & -- \\
 8 & -7  &  7 & -- & -- & -- & -- & --   \\
 9 & -3 & 10 & -- &-- & -- & -- & --  \\
 10 & 1 & 11 & -- & -- & -- & -- & -- \\
  11 & 5 & 20 & -- & -- & -- & -- & -- \\
  12 & 9 & 22 & -- & -- & -- & -- & -- \\
  13 & 13& 24 & -- & -- & -- & --& -- \\
  14 & 17 & 26 & -- & -- & -- & -- & -- \\
  15 & 21 & 28 & 32& -- & -- & -- & --  \\
  16 & 25 & 30 & 33 & 35 & -- & -- & -- \\
  17 & 29 & 34 & -- & 36 & 38 & -- & -- \\
  18 & 33 & 38 & 37 & -- & 39 & 41 & -- \\
  19 & 37 & 42 & 40 & 40 & -- & 42 & 44  \\
  
 \hline
\end{tabular}

\caption{upper bounds of dimensions of strata}
\label{table:2}
\end{center}

\end{table}

\begin{definition}
 A closed subset $X \subset {\mathcal{M}}(H, c_2)$ is called good if every irreducible component of $X$ contains a point
 $[E]$ with $H^2(S, \text{End}^{0}(E)) = 0$, where $\text{End}^{0}(E)$ denotes the traceless endomorphisms of $E$.
\end{definition}\begin{theorem}
Suppose $\mathcal{M}(H, c_2)$ is good for $c_2 \ge 20$.  Then $\overline{\mathcal{M}}(H, c_2)$ is also good for $c_2 \ge 27$.
\end{theorem}
\begin{proof}
Since we are given that $\mathcal{M}(H, c_2)$ is good, every component has expected dimension $4c_2 -39$ for $c_2 \ge 27$. Thus every boundary strata $\mathcal{M}(c_2, c_2^{\prime})$ for $c_2 \ge 27$ and $20 \le c_2^{\prime} \le 26$ has smaller dimension. Also from the   tables \ref{table:1} and \ref{table:2}, we can see that other boundary stratas also have dimension smaller than the expected dimension. Thus every component of $\overline{\mathcal{M}}(H, c_2)$ for $c_2 \ge 27$ has expected dimension, which concludes the Theorem.
\end{proof}  

\section*{Acknowledgement}
 This research was supported in part by SERB-EMR/2015/002172 project fund.

\end{document}